\documentclass[11pt,reqno]{amsart}

\usepackage{stix2}                                          

\usepackage[
    a4paper,
    twoside=false,
    scale=0.75,
]{geometry}                                                 

\usepackage[
    english
]{babel}                                                    

\usepackage{bm}

\usepackage[
    dvipsnames
]{xcolor}                                                   

\usepackage[
    colorlinks=true,
    urlcolor=NavyBlue,
    citecolor=ForestGreen,
    linkcolor=BrickRed,
    hypertexnames=false
]{hyperref}                                                 

\usepackage{accents}                                        

\usepackage[
    shortlabels,
    inline
]{enumitem}                                                 

\usepackage{mathtools}                                      

\usepackage{graphicx}                                       

\usepackage{microtype}                                      

\usepackage{csquotes}                                       
\usepackage[
    bibstyle=ieee,
    citestyle=numeric-comp,
    dashed=true,sorting=nty,
    backref=true,
    maxnames=7,
    backend=bibtex
]{biblatex}                                                 
\DeclareNumChars*{+}                                        

\usepackage{tabularray}                                     

\newcommand{\map}[3]{#1 \colon #2 \to #3 }                  
\newcommand{\wildcard}{
\makebox[\widthof{$x$}][c]{$\cdot$}}                        
\newcommand{\dee}{\mathop{}\!d}
\newcommand{\hrefemail}[1]{\href{mailto:#1}{#1}}            

\newcommand{\mb}[1]{\mathbb{#1}}                             

\newcommand{\ac}[2]{\accentset{#1}{#2}}                     
\newcommand{\ol}[1]{\overline{#1}}

\newcommand{\ceq}{\coloneqq}                                

\newcommand{\R}{ \mb{R} }                                   
\newcommand{\Z}{ \mb{Z} }                                   

\DeclarePairedDelimiter{\parn}{\lparen}{\rparen}            
\DeclarePairedDelimiter{\brac}{\lbrace}{\rbrace}            
\DeclarePairedDelimiter{\brak}{\lbrack}{\rbrack}            
\DeclarePairedDelimiter{\abs}{\lvert}{\rvert}               
\DeclarePairedDelimiter{\rest}{.}{\rvert}                   

\usepackage{amsthm}
\usepackage{thmtools}

\numberwithin{equation}{section}
\theoremstyle{plain}
\newtheorem*{theorem*}{Theorem}
\newtheorem*{maintheorem*}{Main theorem}
\newtheorem{theorem}{Theorem}[section]

\newtheorem{lemma}[theorem]{Lemma}

\theoremstyle{definition}

\theoremstyle{remark}
\newtheorem*{remark*}{Remark}
\newtheorem{remark}[theorem]{Remark}

\usepackage[
    capitalize,
    noabbrev,
    nameinlink
]{cleveref}                                                 

\crefname{assumption}{assumption}{assumptions}              
\crefname{enumi}{part}{parts}                               

\title{Rigidity of symmetric doubly-periodic water waves near shear flows}
\author[D. S. Seth]{Douglas Svensson Seth}
\address{Department of Mathematical Sciences, Norwegian University of Science and Technology, 7491 Trondheim, Norway}
\email{\hrefemail{douglas.s.seth@ntnu.no}}

\author[K. Varholm]{Kristoffer Varholm}
\address{Department of Mathematics, University of Pittsburgh, Pittsburgh, PA 15260}
\email{\hrefemail{kristoffer.varholm@pitt.edu}}

\author[E. Wahlén]{Erik Wahlén}
\address{Centre for Mathematical Sciences, Lund University, PO Box 118, 221 00 Lund, Sweden}
\email{\hrefemail{erik.wahlen@math.lu.se}}

\author[J. Weber]{Jörg Weber}
\address{Faculty of Mathematics, University of Vienna, Oskar-Morgenstern-Platz 1, 1090 Vienna, Austria}
\email{\hrefemail{joerg.weber@univie.ac.at}}

\begin{document}
\begin{abstract}
    We prove that symmetric, doubly periodic, capillary-gravity water waves in finite depth bifurcating from non-uniform non-stagnant shear flows are necessarily two-dimensional to leading order. This is in stark contrast to the case of uniform background flows, where the existence of truly three-dimensional bifurcating waves is well known. While this paper focuses on capillary-gravity water waves, our proof also applies to other suitable types of dynamic boundary conditions, such as hydroelastic waves.
\end{abstract}
\maketitle
\section{Introduction}
The steady capillary-gravity water wave problem consists of the incompressible Euler equations
\begin{gather}
    (\bm{u}\cdot \nabla)\bm{u}+\nabla p+g\bm{e}_3  =0,\label{eq:E1} \\
    \nabla\cdot \bm{u}                             =0\label{eq:E2}
\end{gather}
in the fluid domain
\[
    \Omega^\eta=\brac*{\bm x=(\bm x',x_3)=(x_1,x_2,x_3)\in\R^3:-d<x_3<\eta(\bm x')},
\]
equipped with kinematic boundary conditions
\begin{alignat}{-1}
    & u_3                                                    =0                                &  & \qquad \text{on} \quad & B                              & \ceq \brac{x_3 = -d},\label{eq:kinematic_bottom}       \\
    & u_3         =u_1\partial_1\eta + u_2\partial_2\eta                                       &  & \qquad \text{on} \quad & S^\eta                         & \ceq \brac{x_3 =\eta(\bm x')},\label{eq:kinematic_top}
    \intertext{and the dynamic boundary condition}
    & p+\sigma\nabla\cdot\parn[\bigg]{\frac{\nabla \eta}{\sqrt{1+\vert \nabla\eta\vert^2}}} =0 &  & \qquad \text{on} \quad & S^\eta. \label{eq:dynamic_top}
\end{alignat}

The equations above are written in a frame of reference moving with the wave, such that all quantities are independent of time. The boundary components $B$ and $S^\eta$ refer to the flat rigid bottom and the free surface, respectively. Furthermore, $\bm u=(u_1,u_2,u_3)$, $p$, $g>0$ and $\sigma>0$ denote the velocity field, pressure, gravitational constant, and coefficient of surface tension.

We are particularly interested in doubly-periodic solutions with a particular symmetry, and shall exclusively deal with such flows; see \cref{sec:periodicity_symmetry} for precise definitions. A simple class of trivial solutions to the system~\eqref{eq:E1}--\eqref{eq:dynamic_top} is the one consisting of triples of the form
\[
    (\bm u, p, \eta) = (\ac{0}{\bm{u}},\ac{0}{p},\ac{0}{\eta})
\]
with
\begin{equation}\label{eq:background}
    \begin{tblr}{
            row{1} = {mode = dmath, c},
            row{2} = {mode = text, font=\itshape, c}
        }
        \ac{0}{\bm{u}}(\bm x)=(U(x_3),0,0), & \ac{0}{p}(\bm x)=-gx_3, & \ac{0}{\eta}\equiv 0, \\
        (shear flow)                        & (hydrostatic pressure)  & (flat surface)
    \end{tblr}
\end{equation}
where $U$ is an arbitrary function of the vertical variable $x_3$. Since
\[
    \nabla \times \ac{0}{\bm{u}} = (0,U',0),
\]
such a background flow is irrotational if and only if $U$ is constant, corresponding to uniform flow.

Let us already now state an informal version of the main theorem of this paper. We postpone its precise statement until \cref{thm:main_theorem,rem:dynamic_condition}, as additional exposition is required.

\begin{maintheorem*}
    Consider a trivial solution of the form~\eqref{eq:background} that is both non-uniform and not stagnant. In other words, that $U$ is neither constant nor has any zeros. Then any $C^2$ curve of symmetric $C^2$ solutions to~\eqref{eq:E1}--\eqref{eq:dynamic_top} bifurcating from it consists of waves that are two-dimensional to leading order.

    Moreover, the dynamic boundary condition~\eqref{eq:dynamic_top} can be replaced by any condition ensuring that
    \begin{enumerate*}[(i)]
    \item the kernel of the linearized problem is finite-dimensional, and
    \item nontrivial members of the kernel depend nontrivially on $x_1$.
    \end{enumerate*}
\end{maintheorem*}

\begin{remark*}
    We do \emph{not} rigorously exclude the possibility that truly three-dimensional waves may still appear at higher order. However, we conjecture that this does not occur.
\end{remark*}

The obstruction to obtaining truly three-dimensional solutions comes from a resonance at quadratic order.
The problem with potential resonances in the three-dimensional steady Euler equations has been discussed previously in the literature, especially in the equivalent problem of magnetohydrostatic equilibria. See for example~\cite{weitzner_14_ideal} for a formal construction of a doubly-periodic solution with a fixed boundary. The solution is a small perturbation of a shear flow with rational rotation number with respect to a periodic lattice, and is given by a formal power series. Although a proof of convergence is lacking, it is shown that the equations for the terms in the series can be solved to arbitrary order, despite the presence of resonances. Note, however, that this solution is not symmetric in the sense we will consider here.

Somewhat surprisingly, our result is dramatically different from that of uniform background flows, where $ U $ is constant, since there \emph{is} a curve of truly three-dimensional solutions bifurcating from an open set of uniform flows. In the irrotational setting, three-dimensional doubly-periodic capillary-gravity  waves  ($\sigma > 0$) were first constructed by Reeder \& Shinbrot~\cite{reeder_81_three}. They obtained one-parameter families of smooth small-amplitude solutions bifurcating from uniform flows, with analytic dependence on the amplitude parameter. See also~\cite{sun_93_three} for a related result involving pressure forcing, and~\cite{craig_00_travelling, groves_01_spatial, groves_07_three} for alternative approaches and asymmetric solutions. Doubly-periodic pure gravity waves ($\sigma=0$)  on infinite depth have been constructed by Iooss \& Plotnikov, both with and without symmetry assumptions~\cite{iooss_09_small, iooss_11_asymmetrical}. In contrast to  capillary-gravity waves, these do not come in continuous families, but are only guaranteed to exist for a Cantor set of parameters due to small-divisor issues.

More recently, the existence of symmetric doubly-periodic capillary-gravity waves with small vorticity bifurcating from irrotational uniform background flow was established in~\cite{seth_24_symmetric}. The paper uses a formulation inspired by Lortz' work on magnetohydrostatic equilibria~\cite{lortz_70_uber}, and an adaptation of the classical Crandall--Rabinowitz approach allowing for nonlinear mappings that are Fréchet differentiable with loss of regularity. The solutions and the bifurcation curve can be chosen to have any desired finite degree of regularity.

This raises the question of whether one can find similar symmetric waves bifurcating from other shear flows, as one certainly can in two dimensions~\cite{haziot_22_traveling}. Our result shows that this is not possible ---  at least not through a local bifurcation approach in the spirit of Crandall--Rabinowitz. Such an argument produces a curve of nontrivial solutions bifurcating in the direction of a genuinely doubly-periodic solution to the linearized problem. Nevertheless, this still leaves open the possibility of solutions existing for a Cantor set of parameters, as in the case of pure gravity irrotational waves. Note, however, that the method employed in~\cite{iooss_09_small, iooss_11_asymmetrical} relies on the existence of a sufficiently high order approximate solution, which our result precludes.

Neither does our result rule out the possibility of finding less symmetric solutions bifurcating from general shear flows, or from other types of background flows under a flat surface. Under the special assumption that the velocity $\bm{u}$ is a Beltrami vector field, small rotational solutions have been found bifurcating from a particular family of background flows of the form $(U(x_3),V(x_3),0)$; see~\cite{lokharu_20_existence,groves_24_analytical,seth_24_internal}.
There may also exist weaker solutions; see~\cite{enciso_25_steady} and references therein for constructions of weak solutions to the steady Euler equations with a fixed boundary.

Not much more is currently known about fully three-dimensional waves, except for explicit Gerstner-type solutions for equatorially-trapped waves~\cite{constantin_12_exact,henry_18_three} and edge waves along a sloping beach~\cite{constantin_01_edge}, and several non-existence results for the case of constant vorticity~\cite{wahlen_14_non,chen_23_rigidity,constantin_11_two,martin_22_three}. This is in stark contrast to the well-studied two-dimensional case, where significantly more is known, and with fewer constraints. Instead of providing a detailed list of references here, we simply refer to the recent survey article~\cite{haziot_22_traveling} and the references therein.

The present paper is organized as follows: In \cref{sec:preliminaries} we flatten the free boundary problem~\eqref{eq:E1}--\eqref{eq:dynamic_top} to the fixed domain~$\Omega^0$, and explain the imposed symmetries. We then proceed to introduce asymptotic expansions of all quantities, and present the statement of \cref{thm:main_theorem} in \cref{sec:MainTheorem_expansions}. \Cref{sec:linear_level} concerns the investigation of the linear problem, and characterization of its kernel. Finally, in \cref{sec:quadratic_level} we prove \cref{thm:main_theorem} by establishing a solvability criterion for the equations at the second order.

\section{Preliminaries}\label{sec:preliminaries}
\subsection{Flattening}
As usual, since~\eqref{eq:E1}--\eqref{eq:dynamic_top} is a free boundary problem, we must flatten the a priori unknown fluid domain in order to get an equivalent problem on a fixed domain. There are several approaches commonly used in literature, but here we work with the same one employed in~\cite{seth_24_symmetric}. Specifically, the map $\map{\Phi}{\Omega^0}{\Omega^\eta}$ defined by
\[
    \Phi(\ol{\bm{x}}) \ceq \parn[\big]{\ol{\bm{x}}',\ol{x}_3+\varphi(\ol{\bm{x}})} = \ol{\bm{x}} + \varphi(\ol{\bm{x}}) \bm{e}_3,
\]
where
\[
    \varphi(\ol{\bm{x}}) \ceq\parn[\bigg]{1+\frac{\ol{x}_3}{d}}\eta(\ol{\bm{x}}'),
\]
using the convention that $\ol{\bm{x}} = (\ol{\bm{x}}',\ol{x}_3)$. The Jacobi matrix of this flattening is given by
\[
    J=
    \begin{pmatrix}
        1                  & 0                  & 0    \\
        0                  & 1                  & 0    \\
        \partial_1 \varphi & \partial_2 \varphi & \rho
    \end{pmatrix}=I+\bm{e}_3\otimes\nabla\varphi,
\]
where
\[
    \rho \ceq  1+\partial_3\varphi= 1 + \frac{\eta}{d}= \det(J).
\]

For any scalar field $f$ on $\Omega^\eta$, we define
\[
    \ol{f} \ceq f \circ \Phi
\]
on $\Omega^0$, while vector fields $\bm{v}$ are transformed through the more involved
\[
    J \ol{\bm{v}} \ceq \rho \bm{v} \circ \Phi
\]
instead. In particular, the latter definition ensures that divergence-free vector fields are still divergence-free after flattening.

In the flattened coordinates on
\[
    \Omega \ceq \Omega^0 = \R^2 \times (-d,0),
\]
the incompressible Euler equations~\eqref{eq:E1} and \eqref{eq:E2} are equivalent to
\begin{gather*}
    \frac{1}{\rho} J^\top \parn*{\ol{\bm{u}} \cdot \nabla}\parn*{\frac{1}{\rho}J\ol{\bm{u}}} + \nabla \brak*{\ol{p} + g(\ol{x}_3 + \varphi)}  =0, \\
    \nabla \cdot \ol{\bm{u}}                                                                                                           =0,
\end{gather*}
the dynamic boundary condition~\eqref{eq:dynamic_top} turns into
\begin{alignat*}{-1}
    \ol{p}+\sigma\nabla\cdot\parn[\bigg]{\frac{\nabla \eta}{\sqrt{1+\abs{\nabla\eta}^2}}} & =0 \qquad \text{on} \quad  &  & S \ceq S^0,
    \intertext{while the kinematic boundary conditions~\eqref{eq:kinematic_bottom} and \eqref{eq:kinematic_top} become simply}
    \ol{u}_3                                                                              & = 0 \qquad \text{on} \quad &  & \partial \Omega = S \cup B.
\end{alignat*}

If we now introduce the flattened dynamic pressure
\[
    \wp \ceq \ol{p} + g(\ol{x}_3 + \varphi),
\]
the normalized Jacobi matrix
\begin{equation}
    \label{eq:normalized_jacobi}
    M \ceq \frac{1}\rho J,
\end{equation}
and drop any and all bars, we end up with the equations
\begin{gather}
    M^\top\parn*{ \bm{u}\cdot \nabla}(M\bm{u}) + \nabla \wp = 0, \label{eq:flatE1} \\
    \nabla \cdot \bm{u} = 0 \label{eq:flatE2}
\end{gather}
in $\Omega$, with associated boundary conditions
\begin{alignat}{-1}
    u_3                                                                                          & = 0 \qquad \text{on} \quad &  & \partial \Omega \label{eq:flatKin}, \\
    \wp - g\eta + \sigma\nabla\cdot\parn[\bigg]{\frac{\nabla \eta}{\sqrt{1+\abs{\nabla\eta}^2}}} & =0 \qquad \text{on} \quad  &  & S \label{eq:flatDyn}.
\end{alignat}
We direct the reader to~\cite{seth_24_symmetric} for more details on the flattening transform and the resulting formulas.

\subsection{Periodicity and symmetry}\label{sec:periodicity_symmetry}
Given wavelengths $\lambda_1,\lambda_2 >0$, we introduce the dual lattices $\Lambda$ and $\Lambda^*$ through
\[
    \Lambda \ceq  \lambda_1 \Z \times \lambda_2 \Z, \qquad \Lambda^* \ceq  \kappa_1 \Z \times \kappa_2 \Z,
\]
where $\lambda_i \kappa_i = 2\pi$. Let also $R_j$ denote reflection of the $j$th coordinate for $j=1,2$. We say that a vector field $\bm{v}$ is $(\pm)$-symmetric if it satisfies
\[
    \bm{v}(R_j) = \pm(-1)^j R_j\bm{v},
\]
and that a scalar field $f$ is $(\pm)$-symmetric if
\[
    f(R_j) = (\pm 1)^j f.
\]

We will look for $(+)$-symmetric $\Lambda$-periodic solutions to~\eqref{eq:flatE1}--\eqref{eq:flatDyn}. Spaces of functions that are $(+)$-symmetric and $\Lambda$-periodic will be demarcated with a subscripted $+$. As an example, $C_+^2(\Omega)$ is the space of twice continuously differentiable functions with these properties. We remark that the trivial solutions in~\eqref{eq:background} do not constitute \emph{all} trivial solutions in this symmetry class. The following is pretty much immediate.

\begin{lemma}
    \label{lem:trivial}
    If a vector field $\bm{u}$ is $(+)$-symmetric, $\Lambda$-periodic, and does not depend on $x_1$, then
    \[
        \bm{u}(\bm{x}) = (U(x_2,x_3),0,0),
    \]
    with $U$ being even and $\lambda_2$-periodic with respect to $x_2$. When $\eta \equiv 0$, all such $\bm{u} \in C_+^1(\Omega;\R^3)$ are trivially solutions to~\eqref{eq:flatE1}--\eqref{eq:flatDyn} with corresponding dynamic pressure $\wp \equiv 0$.
\end{lemma}

Note that $(+)$-symmetric $\bm{u}$ entails that  the vorticity $\nabla \times \bm{u}$ is $(-)$-symmetric. Moreover, if we write $\bm{u}$, $\wp$ and $\eta$ in terms of their horizontal Fourier series
\begin{align*}
    \bm{u}(\bm{x}) & = \sum_{\bm{k} \in \Lambda^*} \bm{u}^{\bm{k}}(x_3)e^{i \bm{k} \cdot \bm{x}'}, \\
    \wp(\bm{x})    & = \sum_{k \in \Lambda^*} \wp_{\bm{k}}(x_3)e^{i \bm{k} \cdot \bm{x}'},         \\
    \eta(\bm{x}')  & = \sum_{k \in \Lambda^*} \eta_{\bm{k}}e^{i \bm{k} \cdot \bm{x}'},
\end{align*}
with $\bm{x}'=(x_1,x_2)$, then the coefficients must also be $(+)$-symmetric with respect to $(\bm{k},x_3)$. Therefore, we may also write the series as
\begin{equation}
    \label{eq:fourier_series_trig}
    \begin{aligned}
        \bm{u}(\bm{x}) & = \sum_{\bm{k} \in \Lambda^*}
        \begin{pmatrix*}[r]
            u_1^{\bm{k}}(x_3)\cos(k_1 x_1)\cos(k_2 x_2)  \\
            - u_2^{\bm{k}}(x_3)\sin(k_1 x_1)\sin(k_2 x_2) \\
            i u_3^{\bm{k}}(x_3)\sin(k_1 x_1)\cos(k_2 x_2)
        \end{pmatrix*},                                              \\
        \wp(\bm{x})    & =\sum_{\bm{k} \in \Lambda^*} \wp_{\bm{k}}(x_3)\cos(k_1 x_1)\cos(k_2 x_2), \\
        \eta(\bm{x}')  & = \sum_{\bm{k} \in \Lambda^*} \eta_{\bm{k}}\cos(k_1 x_1)\cos(k_2 x_2).
    \end{aligned}
\end{equation}
Note that all Fourier coefficients corresponding to terms in~\eqref{eq:fourier_series_trig} that vanish trivially due to the presence of a $\sin(0\wildcard)$, also vanish identically by the aforementioned $(+)$-symmetry.

\section{Expansion and main theorem}\label{sec:MainTheorem_expansions}
Let us now, at least formally, write down the asymptotic expansions which we shall consider. We write
\begin{align}
    \eta   & =\epsilon\ac{1}{\eta}+\epsilon^2\ac{2}{\eta}+o(\epsilon^2),\label{eq:expansion_eta} \\
    \wp    & =\epsilon\ac{1}{\wp}+\epsilon^2\ac{2}{\wp}+o(\epsilon^2),\label{eq:expansion_r}
    \intertext{and}
    \bm{u} & =
    \ac{0}{\bm{u}}+
    \epsilon \ac{1}{\bm{u}} + \epsilon^2 \ac{2}{\bm{u}} + o(\epsilon^2),\label{eq:expansion_u}
\end{align}
where we recall the particular form of $\ac{0}{\bm{u}}$ from~\eqref{eq:background}. The triple $(\bm{u},\wp,\eta) = (\ac{0}{\bm{u}},0,0)$ is then a trivial solution of~\eqref{eq:flatE1}--\eqref{eq:flatDyn}.

We can now state a precise version of our main theorem.
\begin{theorem}\label{thm:main_theorem}
    Let $\epsilon\mapsto (\bm u^\epsilon,\wp^\epsilon,\eta^\epsilon)\in C_+^2(\overline{\Omega};\R^3)\times C_+^2(\overline{\Omega})\times C_+^2(\R^2)$ be a $C^2$-curve of solutions to~\eqref{eq:flatE1}--\eqref{eq:flatDyn} in a neighborhood of $\epsilon = 0$, for which
    \[
        (\bm u^0,\wp^0,\eta^0)=(\ac{0}{\bm{u}}, 0, 0),
    \]
    with $\ac{0}{\bm{u}}$  as in~\eqref{eq:background} for some non-constant $U\in C^2([-d,0])$ having no zeros. Then, in the asymptotic expansions~\eqref{eq:expansion_eta}--\eqref{eq:expansion_u}, we necessarily have that
    \[
        \ac{1}{u}_2 \equiv 0 \quad \text{and} \quad \partial_2(\ac{1}{u}_3,\ac{1}{\wp},\ac{1}{\eta}) \equiv 0.
    \]
\end{theorem}

Let us quickly remark that it is not possible to also infer that $\partial_2\ac{1}{u}_1 \equiv 0$ in \Cref{thm:main_theorem} without making further assumptions. This is because we know from \Cref{lem:trivial} that there are additional trivial solutions nearby where ${u}_1$ does in fact depend on $x_2$.

In order to prove \cref{thm:main_theorem}, we progressively investigate~\eqref{eq:flatE1}--\eqref{eq:flatDyn} order by order in $\epsilon$. As previously noted,~\eqref{eq:flatE1}--\eqref{eq:flatDyn} are certainly all satisfied at order $\epsilon^0$ by virtue of $(\smash{\ac{0}{\bm{u}}},0,0)$ being a trivial solution.

\section{Solvability at the linear level}\label{sec:linear_level}

It is convenient to introduce a couple of auxiliary expansions. First,
\[
    \varphi = \epsilon \ac{1}{\varphi} + \epsilon^2 \ac{2}{\varphi} + o(\epsilon^2)
\]
where $\ac{j}{\varphi} = (1+x_3/d)\ac{j}{\eta}$, in terms of which the matrix $M$ in~\eqref{eq:normalized_jacobi} has expansion
\[
    M = I + \epsilon\ac{1}{M} + \epsilon^2 \ac{2}{M} + o(\epsilon^2)
\]
with
\[
    \ac{1}{M} = \bm{e}_3 \otimes \nabla \ac{1}{\varphi}- \partial_3 \ac{1}{\varphi} I,
\]
but where the exact form of $\ac{2}{M}$ will turn out to not matter for our purposes.

The order $\epsilon$ terms of~\eqref{eq:flatE1} give the equation
\[
    (\ac{1}{\bm{u}} \cdot \nabla) \ac{0}{\bm{u}} + (\ac{0}{\bm{u}} \cdot \nabla)\ac{1}{\bm{u}} + \nabla \ac{1}{\wp} + (\ac{0}{\bm{u}} \cdot \nabla)(\ac{1}{M}\ac{0}{\bm{u}}) = 0
\]
or equivalently
\begin{equation}\label{eq:conservation_of_momentum_eps_1}
    \ac{1}{u}_3U'\bm{e}_1+U\partial_1\ac{1}{\bm{u}}+U^2 \partial_1\nabla \times ( \ac{1}{\varphi}\bm{e}_2) + \nabla  \ac{1}{\wp}=0
\end{equation}
after inserting the expressions for $\ac{0}{\bm{u}}$ and $\ac{1}{M}$.
Meanwhile, from incompressibility~\eqref{eq:flatE2} we obtain
\begin{equation}\label{eq:incompressibility_eps}
    \nabla\cdot \ac{n}{\bm{u}}=0,
\end{equation}
while the kinematic boundary condition~\eqref{eq:flatKin} gives
\begin{alignat}{-1}
    \ac{n}{u}_3                                         & =0 & \qquad \text{on} \quad & \partial \Omega \label{eq:kinematic_eps}
    \intertext{for $n\in\brac{1,2}$. Finally, the dynamic boundary condition~\eqref{eq:flatDyn} reads}
    \ac{1}{\wp}-g\ac{1}{\eta}+\sigma\Delta \ac{1}{\eta} & =0 & \qquad \text{on} \quad & S \label{eq:dynamic_condition_eps_1}
\end{alignat}
at order $\epsilon$.

\subsection{Recovering \texorpdfstring{$\protect\ac{1}{\wp}$}{℘¹} from \texorpdfstring{$\protect\ac{1}{\eta}$}{𝜂¹}}
Taking the divergence of~\eqref{eq:conservation_of_momentum_eps_1}, and using~\eqref{eq:incompressibility_eps}, we have
\[
    2U'\partial_1\parn*{\ac{1}{u}_3 + U \partial_1 \ac{1}{\varphi}} + \Delta \ac{1}{\wp} = 0,
\]
or
\[
    \Delta \ac{1}{\wp} - 2\frac{U'}{U} \partial_3 \ac{1}{\wp} = 0,
\]
by taking advantage of the third component of~\eqref{eq:conservation_of_momentum_eps_1}, reading
\begin{equation}
    U \partial_1\parn*{\ac{1}{u}_3 + U\partial_1 \ac{1}{\varphi}} + \partial_3 \ac{1}{\wp} =0.  \label{eq:conservation_of_momentum_eps_1_third}
\end{equation}
Furthermore,~\eqref{eq:conservation_of_momentum_eps_1_third} means we must impose
\begin{equation}
    \label{eq:kinematic_p}
    \begin{aligned}
        \partial_3 \ac{1}{\wp} +U(0)^2 \partial_1^2 \ac{1}{\eta} & =0 \qquad \text{on} \quad S, \\
        \partial_3 \ac{1}{\wp}                                   & = 0 \qquad \text{on} \quad B
    \end{aligned}
\end{equation}
due to the kinematic boundary conditions~\eqref{eq:kinematic_eps}. We also recall that the dynamic boundary condition in~\eqref{eq:dynamic_condition_eps_1} must be satisfied.

Expanding everything in Fourier series with respect to the horizontal variable $\bm{x}'$, we find
\begin{equation}
    \ac{1}{\wp}_{\bm{k}}'' - 2\frac{U'}{U}\ac{1}{\wp}_{\bm{k}}' - \abs{\bm{k}}^2 \ac{1}{\wp}_{\bm{k}} = 0 \label{eq:ode_for_p1_hat}
\end{equation}
for all $\bm{k} \in \Lambda^*$, which can also be written in divergence form
\begin{equation}
    \label{eq:ode_for_p1_hat_divergence}
    \parn[\bigg]{\frac{\ac{1}{\wp}_{\bm{k}}'}{U^2}}' =\abs{\bm{k}}^2 \frac{\ac{1}{\wp}_{\bm{k}}}{U^2},
\end{equation}
with the boundary conditions
\begin{align}
    \ac{1}{\wp}_{\bm{k}}'(0)  & = k_1^2 U(0)^2 \ac{1}{\eta}_{\bm{k}}, \label{eq:kinematic_surface_p1_hat}    \\
    \ac{1}{\wp}_{\bm{k}}'(-d) & = 0, \label{eq:kinematic_bed_p1_hat}
    \intertext{coming from~\eqref{eq:kinematic_p}, and}
    \ac{1}{\wp}_{\bm{k}}(0)   & = (g + \sigma \abs{\bm{k}}^2)\ac{1}{\eta}_{\bm{k}} \label{eq:dynamic_p1_hat}
\end{align}
coming from the dynamic boundary condition~\eqref{eq:dynamic_condition_eps_1}.

We first focus on solving~\eqref{eq:ode_for_p1_hat_divergence} under~\eqref{eq:kinematic_bed_p1_hat}. The first observation is that any constant $\ac{1}{\wp}_0$ is a solution when $\bm{k} = 0$. Secondly, multiplying~\eqref{eq:ode_for_p1_hat_divergence} by $\ac{1}{\wp}_{\bm{k}}$ and integrating by parts shows that
\begin{align*}
    \frac{\ac{1}{\wp}_{\bm{k}}(x_3)\ac{1}{\wp}_{\bm{k}}'(x_3)}{U(x_3)^2} & = \rest*{\frac{\ac{1}{\wp}_{\bm{k}} \ac{1}{\wp}_{\bm{k}}'}{U^2}}_{-d}^{x_3} =\int_{-d}^{x_3} \frac{\abs{\bm{k}}^2 \ac{1}{\wp}_{\bm{k}}(t)^2 + \ac{1}{\wp}_{\bm{k}}'(t)^2}{U(t)^2}\dee t
\end{align*}
for all $x_3 \in [-d,0]$ under the boundary condition~\eqref{eq:kinematic_bed_p1_hat}. From this, we can deduce that if $\bm{k} \neq 0$, then either $\ac{1}{\wp}_{\bm{k}}\equiv 0$, or the solution satisfies
\[
    \ac{1}{\wp}_{\bm{k}}\ac{1}{\wp}_{\bm{k}}' > 0\quad \text{on }(-d,0].
\]
In the latter case, it makes sense to define
\begin{align}\label{eq:qk}
    q_{\bm{k}} \ceq \frac{\ac{1}{\wp}_{\bm{k}}'}{\ac{1}{\wp}_{\bm{k}}},
\end{align}
which then is the positive unique solution to the Riccati equation
\begin{equation}
    \label{eq:qk_riccati}
    q_{\bm{k}}' =2\frac{U'}{U}q_{\bm{k}} + \abs{\bm{k}}^2 - q_{\bm{k}}^2
\end{equation}
on $(-d,0]$, with initial condition $q_{\bm{k}}(-d) = 0$.

Armed with $q_{\bm{k}}$, we infer that the \emph{unique} solution $\ac{1}{\wp}_{\bm{k}}$ to~\eqref{eq:ode_for_p1_hat_divergence} with boundary conditions~\eqref{eq:kinematic_surface_p1_hat} and~\eqref{eq:kinematic_bed_p1_hat} can be recovered as
\begin{equation}\label{eq:recover_r_eps}
    \ac{1}{\wp}_{\bm{k}}= Q_{\bm{k}}\ac{1}{\eta}_{\bm{k}}, \qquad Q_{\bm{k}}(x_3) \ceq \frac{k_1^2 U(0)^2}{q_{\bm{k}}(0)}\exp\parn*{\int_0^{x_3} q_{\bm{k}}(t)\dee t}
\end{equation}
for $\bm{k} \neq 0$. When $\bm{k} = 0$, any constant $\ac{1}{\wp}_0$ is a solution.

Observe next that~\eqref{eq:dynamic_p1_hat} is satisfied for $\bm{k} = 0$ if and only $\ac{1}{\wp}_0 = g\ac{1}{\eta}_0$. For any other $\bm{k} \in \Lambda^*$, it is satisfied if and only if either $\ac{1}{\eta}_{\bm{k}} = 0$ or the \emph{dispersion relation}
\begin{equation}
    \label{eq:dispersion_relation}
    q_{\bm{k}}(0) = \frac{k_1^2 U(0)^2}{g + \sigma \abs{\bm{k}}^2}
\end{equation}
holds. This dispersion relation has at most finitely many solutions, all with $k_1 \neq 0$, due to the following lemma.

\begin{lemma}
    \begin{equation}
        \label{eq:q_limit}
        \lim_{\abs{\bm{k}} \to \infty}{\frac{q_{\bm{k}}(0)}{\abs{\bm{k}}}} = 1.
    \end{equation}
\end{lemma}
\begin{proof}
    Letting
    \begin{equation}
        \label{eq:infimum}
        m \ceq \inf_{x_3 \in [-d,0]}{\frac{U'(x_3)}{U(x_3)}},
    \end{equation}
    the initial value problem
    \[
        l_{\bm{k}}' = 2m l_{\bm{k}} + \abs{\bm{k}}^2 - l_{\bm{k}}^2, \qquad l_{\bm{k}}(-d)=0
    \]
    has the explicit global solution
    \[
        l_{\bm{k}}(x_3) = \frac{\abs{\bm{k}}^2 \tanh\parn*{(x_3+d)\sqrt{m^2 + \abs{\bm{k}}^2}}}{\sqrt{m^2 + \abs{\bm{k}}^2} - m \tanh\parn*{(x_3+d)\sqrt{m^2 + \abs{\bm{k}}^2}}},
    \]
    and since $l_{\bm{k}}$ is then also a sub-solution of~\eqref{eq:qk_riccati}, this results in a lower bound for $q_{\bm{k}}$. In particular
    \[
        \frac{q_{\bm{k}}(0)}{\abs{\bm{k}}} \geq \frac{\abs{\bm{k}}\tanh(d\sqrt{m^2 + \abs{\bm{k}}^2})}{\sqrt{m^2 + \abs{\bm{k}}^2} - m \tanh(d\sqrt{m^2 + \abs{\bm{k}}^2})},
    \]
    Replacing the infimum in~\eqref{eq:infimum} by a supremum, we instead obtain a super-solution and an upper bound. The limit in~\eqref{eq:q_limit} follows immediately from these two bounds.
\end{proof}

We shall also make use of the following property of the dispersion relation. It is however not crucial for the conclusion, and only simplifies a later argument, as we will comment on in \cref{rem:dynamic_condition}.
\begin{lemma}\label{lem:unique}
    If $\bm{k},\bm{l}\in\Lambda^* \setminus \brac{0}$ solve the dispersion relation~\eqref{eq:dispersion_relation}, and $k_1^2 = l_1^2$, then $k_2^2 = l_2^2$ as well.
\end{lemma}
\begin{proof}
    For fixed $k_1$, the left-hand side of~\eqref{eq:dispersion_relation} is clearly strictly increasing in $k_2^2$ by~\eqref{eq:qk_riccati} and the initial condition $q_{\bm{k}}(-d)=0$, while the right-hand side is strictly decreasing. From this, the lemma follows immediately.
\end{proof}

\subsection{Recovering \texorpdfstring{$\protect\ac{1}{\bm{u}}$}{𝒖¹}}

On the Fourier side, conservation of momentum~\eqref{eq:conservation_of_momentum_eps_1} reads
\begin{equation}
    \label{eq:conservation_of_momentum_eps_1_hat}
    \ac{1}{u}_3^{\bm{k}}U'\bm{e}_1 + ik_1 U \ac{1}{\bm{u}}^{\bm{k}} + U^2 ik_1 (i\bm{k},\partial_3) \times (\ac{1}{\varphi}_{\bm{k}}\bm{e}_2) + (i\bm{k},\partial_3)\ac{1}{\wp}_{\bm{k}} = 0
\end{equation}
while incompressibility~\eqref{eq:incompressibility_eps} reads
\begin{equation}
    \label{eq:incompressibility_eps_hat}
    (i\bm{k},\partial_3) \cdot \ac{1}{\bm{u}}^{\bm{k}} = 0
\end{equation}
when $n = 1$, with corresponding boundary conditions
\begin{equation}
    \label{eq:kinematics_eps_hat}
    \ac{1}{u}_3^{\bm{k}}(-d) = \ac{1}{u}_3^{\bm{k}}(0) = 0
\end{equation}
from~\eqref{eq:kinematic_eps}. It is now convenient to split the analysis of these equations into two separate cases.

\subsection*{The case of \texorpdfstring{$k_1 = 0$}{𝑘₁=0}}
By $(+)$-symmetry, cf.~\eqref{eq:fourier_series_trig}, we have $\ac{1}{u}_3^{k_2\bm{e}_2} = \ac{1}{u}_2^{k_2\bm{e}_2} = 0$. Therefore,~\eqref{eq:incompressibility_eps_hat} and~\eqref{eq:kinematics_eps_hat} are trivially satisfied in this case, while~\eqref{eq:conservation_of_momentum_eps_1_hat} reduces to
\[
    k_2 \ac{1}{\wp}_{k_2\bm{e}_2} = \ac{1}{\wp}_{k_2\bm{e}_2}' = 0,
\]
which are satisfied when $k_2 \neq 0$ because $Q_{k_2\bm{e}_2}\equiv0$ in~\eqref{eq:recover_r_eps}. They are also satisfied when $k_2 = 0$, since we saw that $\ac{1}{\wp}_0 = g\ac{1}{\eta}_0$ is constant under~\eqref{eq:dynamic_p1_hat}. No additional demands need to be met by the Fourier coefficients $\ac{1}{u}_1^{k_2 \bm{e}_2}$.

\subsection*{The case of \texorpdfstring{$k_1 \neq 0$}{𝑘₁≠0}}
When $k_1 \neq 0$, we can directly solve~\eqref{eq:conservation_of_momentum_eps_1_hat} to obtain
\[
    \ac{1}{\bm{u}}^{\bm{k}} =\brak*{\frac{(-\bm{k},i\partial_3)}{Uk_1}-\frac{U'\bm{e}_1}{U^2k_1^2}\partial_3}\ac{1}{\wp}_{\bm{k}} - (i\bm{k},\partial_3) \times (U \ac{1}{\varphi}_{\bm{k}}\bm{e}_2),
\]
under which~\eqref{eq:incompressibility_eps_hat} becomes
\[
    \frac{i}{k_1 U}\parn*{\ac{1}{\wp}_{\bm{k}}'' -2 \frac{U'}{U} \ac{1}{\wp}_{\bm{k}}' -\abs{\bm{k}}^2 \ac{1}{\wp}_{\bm{k}}} = 0,
\]
which is satisfied by~\eqref{eq:ode_for_p1_hat}. Moreover, since
\[
    \ac{1}{u}_3^{\bm{k}} = \frac{i}{k_1 U} \parn*{\ac{1}{\wp}_{\bm{k}}' - k_1^2 U^2 \ac{1}{\varphi}_{\bm{k}}},
\]
the boundary conditions in~\eqref{eq:kinematics_eps_hat} are satisfied by~\eqref{eq:kinematic_surface_p1_hat} and~\eqref{eq:kinematic_bed_p1_hat}.

\subsection{Summing up at the linear level}
We now sum up the findings at the linear level. Let $N(U)$ be the set of $\bm{k} \in \Lambda^* \setminus \brac{0}$ satisfying the dispersion relation~\eqref{eq:dispersion_relation}. Then the kernel of the linearization consists precisely of the elements of the form
\begin{equation}
    \label{eq:representation_kernel}
    \begin{aligned}
        \ac{1}{\eta}(\bm{x}')  & = a_0 + \sum_{\bm{k} \in N(U)} a_{\bm{k}} e^{i \bm{k} \cdot \bm{x}'},                                                                                                                                        \\
        \ac{1}{\wp}(\bm{x})    & = ga_0 + \sum_{\bm{k} \in N(U)} a_{\bm{k}}Q_{\bm{k}}(x_3) e^{i \bm{k} \cdot \bm{x}'},                                                                                                                        \\
        \ac{1}{\bm{u}}(\bm{x}) & = - \nabla \times (U \ac{1}{\varphi} \bm{e}_2) + \ac{1}{\bm{v}}(\bm{x}),                                                                                                                                     \\
        \ac{1}{\bm{v}}(\bm{x}) & = w(x_2,x_3)\bm{e_1} + \sum_{\bm{k} \in N(U)} a_{\bm{k}}Q_{\bm{k}}(x_3)\brak*{\frac{(-\bm{k},iq_{\bm{k}}(x_3))}{U(x_3)k_1}-\frac{q_{\bm{k}}(x_3)U'(x_3)\bm{e}_1}{k_1^2 U(x_3)^2}}e^{i \bm{k} \cdot \bm{x}'},
    \end{aligned}
\end{equation}
where $w$ is $\lambda_2$-periodic and even in $x_2$, and the $a_{\bm{k}}$ are $(+)$-symmetric with respect to $\bm{k}$.

\section{Solvability at the quadratic level and proof of the main theorem}\label{sec:quadratic_level}
Let us now turn to the quadratic level. The terms in~\eqref{eq:flatE1} of order $\epsilon^2$ give
\begin{multline}
    \label{eq:conservation_of_momentum_eps_2}
    \parn{\ac{0}{\bm{u}} \cdot \nabla}\parn[\big]{\ac{2}{\bm{u}} + \ac{1}{M}\ac{1}{\bm{u}} + \ac{2}{M}\ac{0}{\bm{u}}} + \parn{\ac{1}{\bm{u}} \cdot \nabla}\parn[\big]{\ac{1}{\bm{u}} + \ac{1}{M}\ac{0}{\bm{u}}} + \parn{\ac{2}{\bm{u}} \cdot \nabla}\ac{0}{\bm{u}}\\
    + \ac{1}{M}^\top\brak[\Big]{(\ac{0}{\bm{u}}\cdot\nabla)\parn[\big]{\ac{1}{\bm{u}}+\ac{1}{M}\ac{0}{\bm{u}}} + (\ac{1}{\bm{u}}\cdot \nabla)\ac{0}{\bm{u}}} + \nabla \ac{2}{\wp} = 0,
\end{multline}
where we compute the following:
\begin{align*}
    \parn{\ac{0}{\bm{u}} \cdot \nabla}\parn[\big]{\ac{2}{\bm{u}} + \ac{1}{M}\ac{1}{\bm{u}} + \ac{2}{M}\ac{0}{\bm{u}}}                                      & = U \partial_1\parn[\big]{\ac{2}{\bm{u}} + \ac{1}{M}\ac{1}{\bm{u}} + \ac{2}{M}\ac{0}{\bm{u}}},                                                                \\
    \parn{\ac{1}{\bm{u}} \cdot \nabla}\parn[\big]{\ac{1}{\bm{u}} + \ac{1}{M}\ac{0}{\bm{u}}}                                                                & =
    \begin{multlined}[t]\brak[\Big]{\partial_3(U \ac{1}{\varphi})\partial_1(U'\ac{1}{\varphi}) - \partial_1(U \ac{1}{\varphi})\partial_3(U'\ac{1}{\varphi}) + \ac{1}{\bm{v}}\cdot \nabla (U'\ac{1}{\varphi})}\bm{e}_1\\
        +\partial_3(U \ac{1}{\varphi})\partial_1 \ac{1}{\bm{v}} - \partial_1(U \ac{1}{\varphi})\partial_3 \ac{1}{\bm{v}} + (\ac{1}{\bm{v}}\cdot \nabla)\ac{1}{\bm{v}},
    \end{multlined}                                                                                      \\
    (\ac{2}{\bm{u}} \cdot \nabla)\ac{0}{\bm{u}}                                                                                                            & =\ac{2}{u}_3U'\bm{e}_1                    ,                                                                                                                   \\
    \ac{1}{M}^\top\brak[\Big]{(\ac{0}{\bm{u}}\cdot\nabla)\parn[\big]{\ac{1}{\bm{u}}+\ac{1}{M}\ac{0}{\bm{u}}} + (\ac{1}{\bm{u}}\cdot \nabla)\ac{0}{\bm{u}}} & = - \partial_3 \ac{1}{\varphi}U'\ac{1}{v}_3\bm{e}_1 + U \partial_1 \ac{1}{v}_3 \nabla \ac{1}{\varphi} -U \partial_3 \ac{1}{\varphi}\partial_1 \ac{1}{\bm{v}}.
\end{align*}
In particular, taking the first component of the curl of~\eqref{eq:conservation_of_momentum_eps_2}, we have
\begin{multline*}
    \parn*{\nabla \times \brak[\big]{\parn{\ac{1}{\bm{v}} \cdot \nabla}\ac{1}{\bm{v}}}}_1 + \partial_1\parn[\big]{\nabla \times \brak[\big]{U\parn[\big]{\ac{2}{\bm{u}} + \ac{1}{M}\ac{1}{\bm{u}} + \ac{2}{M}\ac{0}{\bm{u}}}}}_1 \\
    +\partial_2\parn*{\partial_3(U\ac{1}{\varphi})\partial_1 \ac{1}{v}_3 - \partial_1(U\ac{1}{\varphi})\partial_3 \ac{1}{v}_3} - \partial_3\parn*{\partial_3(U\ac{1}{\varphi})\partial_1\ac{1}{v}_2 - \partial_1(U\ac{1}{\varphi})\partial_3 \ac{1}{v}_2} \\
    + \partial_3(\partial_3\ac{1}{\varphi}U\partial_1 \ac{1}{v}_2 - \partial_2 \ac{1}{\varphi} U \partial_1 \ac{1}{v}_3) = 0.
\end{multline*}

Averaging over one period in the $x_1$-direction we see that
\begin{align*}
    \intbar_0^{\lambda_1} \partial_2\parn*{\partial_3(U\ac{1}{\varphi})\partial_1 \ac{1}{v}_3 - \partial_1(U\ac{1}{\varphi})\partial_3 \ac{1}{v}_3}\dee x_1 & = \intbar_0^{\lambda_1} \partial_2\parn*{\partial_3(U\ac{1}{\varphi})\partial_1 \ac{1}{v}_3 + (U\ac{1}{\varphi})\partial_1\partial_3 \ac{1}{v}_3}\dee x_1 \\
    & = \partial_3\intbar_0^{\lambda_1} \partial_2 ( \ac{1}{\varphi}U\partial_1 \ac{1}{v}_3)\dee x_1,                                                           \\
    \intbar_0^{\lambda_1}\partial_3\parn*{\partial_3(U\ac{1}{\varphi})\partial_1\ac{1}{v}_2 - \partial_1(U\ac{1}{\varphi})\partial_3 \ac{1}{v}_2}\dee x_1   & =\intbar_0^{\lambda_1} \partial_3\parn*{\partial_3(U\ac{1}{\varphi})\partial_1\ac{1}{v}_2 + U\ac{1}{\varphi}\partial_1\partial_3 \ac{1}{v}_2}\dee x_1     \\
    & =\partial_3\intbar_0^{\lambda_1} \partial_3 (\ac{1}{\varphi}U\partial_1\ac{1}{v}_2)\dee x_1
\end{align*}
and that
\[
    \partial_2 ( \ac{1}{\varphi}U\partial_1 \ac{1}{v}_3) - \partial_3 (\ac{1}{\varphi}U\partial_1\ac{1}{v}_2) + \partial_3\ac{1}{\varphi}U\partial_1 \ac{1}{v}_2 - \partial_2 \ac{1}{\varphi} U \partial_1 \ac{1}{v}_3 = \ac{1}{\varphi}[\nabla \times (U \partial_1 \ac{1}{\bm{v}})]_1=0
\]
by~\eqref{eq:conservation_of_momentum_eps_1}, which reads
\[
    \ac{1}{v}_3 U' \bm{e}_1 + U \partial_1 \ac{1}{\bm{v}} + \nabla \ac{1}{\wp} = 0
\]
for $\bm{v}$. Hence
\[
    \intbar_0^{\lambda_1} \parn[\big]{\nabla \times \brak[\big]{\parn[\big]{\ac{1}{\bm{v}} \cdot \nabla}\ac{1}{\bm{v}}}}_1 \dee x_1 = 0
\]
is a necessary condition for solvability at the second order. Since $\ac{1}{\bm{v}}$ is divergence free by~\eqref{eq:representation_kernel} and \eqref{eq:incompressibility_eps}, this condition simplifies to
\begin{align}
    0 & = \intbar_0^{\lambda_1} \parn*{\partial_2\brak[\big]{\ac{1}{v}_1 \partial_1 \ac{1}{v}_3 + \ac{1}{v}_2\partial_2 \ac{1}{v}_3 + \ac{1}{v}_3 \partial_3 \ac{1}{v}_3} - \partial_3\brak[\big]{\ac{1}{v}_1 \partial_1 \ac{1}{v}_2 + \ac{1}{v}_2 \partial_2 \ac{1}{v}_2 + \ac{1}{v}_3 \partial_3 \ac{1}{v}_2}} \dee x_1 \notag                                                      \\
    & = \intbar_0^{\lambda_1} \parn*{\partial_2\brak[\big]{(\partial_2 \ac{1}{v}_2 + \partial_3 \ac{1}{v}_3)\ac{1}{v}_3 + \ac{1}{v}_2\partial_2 \ac{1}{v}_3 + \ac{1}{v}_3 \partial_3 \ac{1}{v}_3} - \partial_3\brak[\big]{(\partial_2 \ac{1}{v}_2 + \partial_3 \ac{1}{v}_3) \ac{1}{v}_2 + \ac{1}{v}_2 \partial_2 \ac{1}{v}_2 + \ac{1}{v}_3 \partial_3 \ac{1}{v}_2}} \dee x_1 \notag \\
    & = \intbar_0^{\lambda_1}\parn*{\brak{\partial_2^2 - \partial_3^2}(\ac{1}{v}_2 \ac{1}{v}_3) + \partial_2 \partial_3 (\ac{1}{v}_3^2 - \ac{1}{v}_2^2)} \dee x_1, \label{eq:solvability_integral}
\end{align}
into which we will insert the representation of $\ac{1}{\bm{v}}$ from~\eqref{eq:representation_kernel}.

To that end, let us first introduce the subset
\[
    N(U)^+ \coloneqq \brac{\bm{k} \in N(U):k_1>0,k_2>0}
\]
of $N(U)$. Then, we see that
\begin{align*}
    \ac{1}{v}_2 & = 4\sum_{\bm{k} \in N(U)^+}  \frac{a_{\bm{k}} k_2 Q_{\bm{k}}}{k_1U} \sin(k_1 x_1) \sin(k_2 x_2)
    \intertext{and}
    \ac{1}{v}_3 & = - 2 \sum_{
        \substack{
            k_1 \bm{e}_1 \in N(U)                                                                                         \\
            k_1 > 0
    }} \frac{a_{\bm{k}} Q_{\bm{k}} q_{\bm{k}}}{k_1 U} \sin(k_1 x_1) - 4\sum_{\bm{k} \in N(U)^+} \frac{a_{\bm{k}} Q_{\bm{k}} q_{\bm{k}}}{k_1 U} \sin(k_1 x_1)\cos(k_2 x_2),
\end{align*}
whence, using \cref{lem:unique},
\begin{align*}
    \intbar_0^{\lambda_1} \ac{1}{v}_2 \ac{1}{v}_3 \dee x_1  & = - 4 \sum_{\bm{k} \in N(U)^+} \frac{a_{\bm{k}}^2 k_2 Q_{\bm{k}}^2 q_{\bm{k}}}{k_1^2 U^2} \sin(2k_2 x_2), \\
    \partial_2 \intbar_0^{\lambda_1} \ac{1}{v}_3^2 \dee x_1 & = -8 \sum_{\bm{k} \in N(U)^+} \frac{a_{\bm{k}}^2 k_2 Q_{\bm{k}}^2q_{\bm{k}}^2}{k_1^2U^2}\sin(2k_2 x_2),   \\
    \partial_2 \intbar_0^{\lambda_1} \ac{1}{v}_2^2 \dee x_1 & = 8 \sum_{\bm{k} \in N(U)^+} \frac{a_{\bm{k}}^2 k_2^3 Q_{\bm{k}}^2}{k_1^2 U^2}\sin(2k_2x_2),
\end{align*}
from which we deduce that
\begin{align*}
    0 & = \intbar_0^{\lambda_1}\parn*{\brak{\partial_2^2 - \partial_3^2}(\ac{1}{v}_2 \ac{1}{v}_3) + \partial_2 \partial_3 (\ac{1}{v}_3^2 - \ac{1}{v}_2^2)} \dee x_1 \\
    & = -8 \frac{U'}{U^3} \sum_{\bm{k} \in N(U)^+} \frac{a_{\bm{k}}^2 k_2}{k_1^2}Q_{\bm{k}}^2(k_1^2 -k_2^2 + q_{\bm{k}}^2)\sin(2k_2 x_2)
\end{align*}
after repeated use of the Riccati equation~\eqref{eq:qk_riccati} that $q_{\bm{k}}$ satisfies.

In particular, taking the derivative with respect to $x_2$ and evaluating at $x_2 = 0$, we arrive at the identity
\begin{equation}
    \label{eq:final_identity}
    U'f = 0
\end{equation}
on $[-d,0]$, where $f$ is the function given by
\[
    f \ceq \sum_{\bm{k} \in N(U)^+} \frac{a_{\bm{k}}^2 k_2^2}{k_1^2}Q_{\bm{k}}^2(k_1^2-k_2^2 + q_{\bm{k}}^2).
\]

\begin{lemma}
    \label{lem:proof_of_main_theorem}
    If there is any $\bm{k} \in N(U)^+$ for which $a_{\bm{k}} \neq 0$, then $U$ is necessarily constant.
\end{lemma}
\begin{proof}
    Observe first that we can compute
    \[
        f' = 4\sum_{\bm{k} \in N(U)^+} \frac{a_{\bm{k}}^2 k_2^2}{k_1^2}Q_{\bm{k}}^2 q_{\bm{k}}\parn*{k_1^2 + \frac{U'}{U}q_k}
    \]
    by using~\eqref{eq:qk_riccati}. Define now the set
    \[
        V = \brac{x_3 \in (-d,0) : U'(x_3) = 0},
    \]
    which clearly a closed subset of $(-d,0)$.

    We first demonstrate that $V$ is nonempty. Indeed, the initial condition $q_{\bm{k}}(-d) = 0$, combined with continuity, implies that there is a nonempty interval $I_\delta = (-d,-\delta)$ such that
    \[
        k_1^2 + \frac{U'}{U}q_{\bm{k}} \geq \frac{1}{2}k_1^2
    \]
    for all $x_3 \in I_\delta$ and $\bm{k} \in N(U)^+$. It follows that
    \[
        f'(x_3) \geq 2 \sum_{\bm{k} \in N(U)^+} a_{\bm{k}}^2 k_2^2Q_{\bm{k}}(x_3)^2 q_{\bm{k}}(x_3) > 0
    \]
    for all $x_3 \in I_\delta$, as all terms for which $a_{\bm{k}} k_2 \neq 0$ are positive. By possibly shrinking $\delta$, we may therefore ensure that $f$ is nonzero on $I_\delta$. Due to~\eqref{eq:final_identity}, this means that $I_\delta \subset V$, which in particular is therefore nonempty.

    Finally, suppose that $z \in V$. Then
    \[
        f'(z) = 4 \sum_{\bm{k} \in N(U)^+} a_{\bm{k}}^2 k_2^2 Q_{\bm{k}}(z)^2 q_{\bm{k}}(z) > 0,
    \]
    so there is some open interval $I \ni z$ such that $f$ is nonzero on the punctured neighborhood $I \setminus \brac{z}$. Due to~\eqref{eq:final_identity}, this means that $U'$ must necessarily vanish on $I \setminus \brac{z}$, and therefore all of $I$. Hence $I \subset V$, and we conclude that $V = (-d,0)$.
\end{proof}

This concludes the proof of \cref{thm:main_theorem}.

\begin{remark}\label{rem:dynamic_condition}
    Recapitulating \cref{sec:linear_level,sec:quadratic_level}, we note that the specific dynamic boundary condition~\eqref{eq:dynamic_top} was only used to ensure that $N(U)$ is finite (so that we only have to deal with finite sums) and that $k_1\ne0$ for all $\bm k\in N(U)$ (so that we can divide by $k_1$ in the computations). In particular, this means that \cref{thm:main_theorem} still holds for other physical dynamical boundary conditions, as long as these properties continue to be satisfied. One example of this would be hydroelastic waves, like those treated in~\cite{plotnikov_11_modelling, ahmad_24_resonant}.

    The additional property we deduced from the dynamic boundary condition in \cref{lem:unique}, while mild, is in fact not necessary for the conclusion of \cref{thm:main_theorem}. It was introduced mainly to simplify the presentation: Instead of appealing to \cref{lem:unique}, we could have dealt with sums containing potentially more terms appearing when inserting the representation of $\smash{\ac{1}{\bm{v}}}$ into~\eqref{eq:solvability_integral}. This can be simplified somewhat by projecting onto any given mode $\sin(2k_2 x_2)$ associated with a nontrivial coefficient. Then, similar computations would lead to the same conclusion as in \Cref{lem:proof_of_main_theorem}.
\end{remark}

\section*{Acknowledgments}

All four authors were supported by the Swedish Research Council under grant no. 2021-06594 while in residence at Institut Mittag-Leffler in Djursholm, Sweden during the fall semester of 2023.

This project has received funding from the European Research Council (ERC) under the European Union's Horizon 2020 research and innovation programme (grant agreement no. 678698), the Swedish Research Council (grant no. 2020-00440), and by the Research Council of Norway (grant agreement
No. 325114).

This research was funded in whole or in part by the Austrian Science Fund (FWF) [10.55776/ ESP8360524]. For open access purposes, the author has applied a CC BY public copyright license to any author-accepted manuscript version arising from this submission.

\printbibliography[title=References]
\end{document}